\documentclass[12pt]{amsart}
\usepackage[english]{babel}
\usepackage{graphicx}
\usepackage{amsfonts}
\usepackage{amsmath}
\usepackage{amssymb}
\usepackage{amscd}
\usepackage[all,cmtip,matrix, arrow]{xy}
\usepackage{amsthm}
\usepackage{cmap}
\usepackage{tikz-cd}
\usepackage[left=3cm,right=3cm,
    top=2.5cm,bottom=2.5cm]{geometry}

\newtheorem{lemma}{Lemma}
\newtheorem{theorem}[lemma]{Theorem}
\newtheorem{proposition}[lemma]{Proposition}
\newtheorem{obs}[lemma]{Observation}

\newtheorem{corollary}[lemma]{Corollary}

\theoremstyle{definition}
\newtheorem{definition}[lemma]{Definition}

\title{On the Characteristic Foliation on a Smooth Hypersurface in a Holomorphic Symplectic Fourfold}
\author{E. Amerik, L. Guseva}
\date{}

\begin{document}
\maketitle
\section{Introduction}

\begin{definition} A holomorphic symplectic manifold is a simply-connected compact K{\"a}hler 
manifold $X$ with a nowhere degenerate global holomorphic two-form $\omega$.
A holomorphic symplectic manifold is irreducible if in addition $H^0(X, \Omega^2)$ is spanned by
such an $\omega$.
\end{definition}

This terminology is explained by Bogomolov decomposition theorem which states that, up to a finite
\'etale covering, each holomorphic symplectic manifold is a product of several irreducible ones
and a torus.

Let X be a holomorphic symplectic manifold with a holomorphic symplectic form $\omega.$ 
Let $D$ be a smooth divisor on $X.$ At each point of $D$, the restriction of $\omega$ to $D$ has one-dimensional kernel. This gives a non-singular foliation $\mathcal{F}$ on $D$, called the characteristic foliation. 
Hwang and Viehweg in \cite{HV} have shown that if $X$ is projective and $D$ is of general type, then 
$\mathcal{F}$ cannot be algebraic (unless in the trivial case when $X$ is a surface, $D$ is a curve,
so $\mathcal{F}$ has a single leaf equal to $D$; recall that a foliation in curves on a compact
K\"ahler manifold is called algebraic when all
its leaves are compact complex curves). This result has been extended by Amerik and Campana 
in \cite{AC} to the case when $D$ is not necessarily of general type; in particular,
when $X$ is irreducible, they proved that the characteristic foliation $\mathcal{F}$ on a nonsingular irreducible divisor $D$ is algebraic if and only if the leaves of $\mathcal{F}$ are rational curves or $X$ is a surface.

Suppose now that $X$ is an irreducible holomorphic symplectic fourfold and $D$, $\mathcal{F}$
are as above. It is easy to
give an example when the Zariski closure of a general leaf of $\mathcal{F}$ is a surface. Indeed
such is the case when $X$ has a Lagrangian fibration $f: X\to Z$ and $D$ is the preimage of a general curve $C\subset Z.$ The leaves of $\mathcal{F}$ are contained in the fibers of $f$; note that the 
general
fiber is a torus by the Arnold-Liouville theorem. 
The aim of this paper is to prove that all examples where the Zariski closure of a general leaf
is two-dimensional are obtained in this way. 
 
\begin{theorem} \label{30}
Let $X$ be an irreducible holomorphically symplectic 4-dimensional manifold and let $D$ be an irreducible smooth divisor on X. Suppose that a general leaf of the characteristic foliation $\mathcal{F}$ on $D$ is non-compact but there exists a rational fibration $p: D\dasharrow C$ by surfaces such that every leaf of $\mathcal{F}$ is contained in the closure of some fiber of $p.$ Then there exists an almost holomorphic lagrangian fibration
$f: X\dasharrow B$ extending $p$. In particular the general fiber of $p$ is a torus.
\end{theorem}

Greb, Lehn and Rollenske in \cite{GLR} in the non-projective case, Amerik in \cite{A} in the 4-dimensional case and Hwang and Weiss in \cite{HW} in the general case proved that 
any Lagrangian torus in an irreducible holomorphically symplectic manifold is a fibre of an almost holomorphic Lagrangian fibration. The general fiber of $p$ is clearly
lagrangian, since the tangent space to $S$ at a general point is contained in the
tangent space to $D$ and contains the kernel of the restriction of the symplectic 
form to $D$.  Therefore to prove theorem \ref{30} it 
suffices to show that the general fiber $S$ of $p$ is a torus.

Our argument proceeds as follows. In the next section, using the classification of 
surfaces and Brunella's results on foliations on surfaces, we reduce the problem to 
the case when $S$ is an elliptic fibration with some special 
properties. Then we remark that such an $S$ deforms away from $D$ and its general deformation has similar 
properties. Finally, in the last section we produce a deformation-theoretic 
argument leading to the conclusion.

\section{Smooth foliations on surfaces and the geometry of $S$}

We assume that the closure of a sufficiently general leaf of the characteristic 
foliation $\mathcal{F}$ is a surface. There is then a family of such surfaces which
rationally fibers $D$; we denote a general surface of the family by $S$ and the rational
fibration by $p:D\dasharrow B$. Apriori $p$ can have an indeterminacy locus and
$S$ can be singular along this locus. This indeterminacy locus is also the intersection
of general surfaces in the family (notice that if $p$ has indeterminacy then $B$ is 
a rational curve, since certain rational curves coming from the resolution of 
indeterminacy then dominate $B$; so $S$ is a general member of a linear system defining 
$p$, therefore the base locus of the linear system is the intersection of two such $S$).   

From the fact that $\mathcal{F}$ is a 
smooth foliation
(i.e. in the neighbourhood of each point $D$ looks like a product with a curve)
given by the $\omega$-orthogonal to $D$ we deduce an obvious
\begin{obs}\label{lagr}
$S$ is a lagrangian surface, the indeterminacy  locus of $p$ is a union of leaves
of $\mathcal{F}$, and $\mathcal{F}$ induces a smooth foliation on the normalization
of $S$.
\end{obs}

Note that lagrangian implies projective in this context (see e.g. \cite{C}).
Our next aim is to show that $p$ is regular and hence $S$ is smooth.



Let $\tau: S'\to S$ be the normalization. The image of $T_{S'}$ in $\tau^*T_D$ contains
$\tau^*{\mathcal{F}}$. We denote by $\mathcal{F}'$ the induced foliation on $S'$.
\begin{lemma}\label{conormal}
The conormal bundle of $\mathcal{F}'$ is effective.
\end{lemma}

\begin{proof} Consider the exact sequence 
$$0\to T_{S'}/\tau^*{\mathcal{F}}\to \tau^*(T_D/\mathcal{F})\to \tau^*(T_D)/T_{S'}\to 0.$$
The sheaf in the middle carries a symplectic form and thus, being isomorphic to its 
dual, has zero determinant. We have the normal
bundle of $\mathcal{F}'$ on the left, and the first Chern class of the sheaf on the
right is therefore equal to that of the conormal bundle in the sense of algebraic 
cycles, but the sheafs itself are not necessarily isomorphic since $\tau^*(T_D)/T_{S'}$ can have
torsion at the critical points of $\tau$. The set of those critical points is a
divisor on $S'$ (it is a union of leaves of the regular foliation $\mathcal{F'}$),
so the conormal bundle of $\mathcal{F'}$ is the torsion-free part of $\tau^*(T_D)/T_{S'}$
 possibly tensored up with something effective. But $\tau^*(T_D)/T_{S'}$ modulo torsion
is already effective by deformation theory, since $\tau$ is a general member of a 
one-dimensional family of maps to $D$ (see e.g. \cite{HM}, p. 108-110).
\end{proof}

This lemma and the theory of foliation of surfaces imply the needed

\begin{proposition}\label{regular}
The map $p$ is regular, therefore its general fiber $S$ is smooth.
\end{proposition}

\begin{proof} Brunella in  \cite{Br-book}[Proposition 6.2]  shows that if 
$h^{0}(X, N_{\mathcal{F}}^{*})\ge 1$ for a foliation $\mathcal{F}$ on a surface X and
$C$ a compact curve invariant by the foliation, then either $C$ is contractible or
$\mathcal{F}$ is a fibration over a non-rational curve. If $p$ is not regular, apply this to
$X=S'$ and $C$ the compact leaf of the foliation coming from the indeterminacy locus of 
$p$. By Camacho-Sad formula $C^2=0$, so $C$ is not contractible, but $\mathcal{F'}$ is
not a fibration either, a contradiction.
\end{proof}

From now on we consider the restriction of the foliation $\mathcal{F}$ on $D$ to the smooth surface $S$, which we denote by the same letter $\mathcal{F}$. By assumption, $\mathcal{F}$ has at least one non-algebraic leaf on $S$. As it has been already remarked in the proof
of lemma \ref{conormal}, the symplectic form induces an isomorphism between $N_{S/D}$ and the conormal bundle of $\mathcal{F}$. Since $N_{S/D}$ is trivial, so is 
$N_{\mathcal F}$ and we have  

\begin{lemma}\label{12}
$N_{\mathcal{F}}$ is the trivial line bundle, $K_{S}\simeq K_{\mathcal F}=(T_{\mathcal{F}})^{*}$
and $c_2(S)=0$.
\end{lemma}
\begin{proof} Indeed the last two statements follow from the first by the exact sequence
of a smooth foliation
$$0\to T_{\mathcal{F}}\to T_{S}\to N_{\mathcal{F}} \to 0.$$ 

\end{proof}

\begin{corollary}\label{3} $S$ is not of general type and the minimal model of $S$ can not be a $K3$ surface or an Enriques surface. If the minimal model of $S$ is a torus, a 
bielliptic or a properly elliptic surface, then $S$ is itself minimal.
\end{corollary}
\begin{proof}
The first statement is \cite{BPV} [Proposition VII.2.4]; the rest follows from the 
classification and from the fact
that $c_2$ goes up under blow-ups. Indeed  $c_2$ is strictly positive for K3 and Enriques
surfaces and non-negative for the other minimal models mentioned.
\end{proof}

\begin{lemma} \label{15}
$h^0(K_{S})>0$, so the Kodaira dimension of $S$ is not equal to $-\infty$ and $S$ can not be a bielliptic surface or its blow-up.
\end{lemma}
\begin{proof}
This is an immediate consequence of a lemma by Brunella \cite[Lemma 7]{Br1} which affirms 
that a regular foliation $\mathcal{F}$ on a smooth projective surface $S$ with 
$h^{0}(T_{\mathcal{F}}^{*})=0$ and $h^{0}(N_{\mathcal{F}})>0$ is a fibration. Since this is not our
case we must have $h^0(K_{S})=h^{0}(K_{\mathcal{F}})>0$.
\end{proof}

By the Kodaira-Enriques classification, we see that if $S$ is not a torus (in which case
we are done), then $S$ is properly elliptic, that is, $S$ is an elliptic surface of
Kodaira dimension 1. We denote the elliptic fibration by $\pi: S\to C$. Note that 
the only singularities of $\pi$ are multiple fibers, by the fact that $c_2(S)=0$ 
(see e.g. \cite{BPV}). Therefore $\pi$ 
induces another smooth rank-one foliation on $S$ which we denote by $\mathcal{G}$ and the corresponding map by $\pi: S\to C$. Let us denote by $g$ the genus of $C$. Note that the arithmetic genus $\chi(\mathcal{O}_{S})$ of $S$ is equal to zero by Noether formula.

 No fiber of $\pi$ can be $\mathcal{F}$-invariant. This can be seen using the same 
proposition by Brunella as in the proof of proposition \ref{regular}. Indeed the fibers of $\pi$ are not contractible,
and $\mathcal{F}$ is a non-algebraic foliation with effective conormal bundle.
Let us now consider the divisor of tangency $D_{tan}$ between $\mathcal{F}$ and 
$\mathcal{G}$. By definition, the support of $D_{tan}$ consists of points where $\mathcal{F}$ and $\mathcal{G}$ are not transversal (see \cite[page 573]{Br1}). 
\begin{lemma}
The divisor of tangency $D_{tan}$ is trivial. 
\end{lemma}
\begin{proof}
Suppose that there exists a point of tangency $p.$  Let $E$ be the leaf of $\mathcal{G}$
(i.e. the reduction of the fiber of $\pi$)
containing $p$. By the claim above $E$ is not $\mathcal{F}$-invariant, so we can consider the tangency index $tang(E,\mathcal{F})$ (see e.g. \cite{Br-book}, chapter III) and by supposition it is strictly positive. But by a formula from the same reference $$tang(E,\mathcal{F})=c_{1}(N_{\mathcal{F}})\cdot E-\chi(E)=0,$$ hence we have a contradiction.
\end{proof}

\begin{corollary} \label{50}
The tangent bundle to $S$ is the direct sum of its subbundles $T_{\mathcal{F}}$ and
$T_{\mathcal{G}}$, and  $T_{\mathcal{G}}$ is trivial. 
\end{corollary}
\begin{proof} The first statement is another formulation of the transversality of 
${\mathcal{F}}$ and ${\mathcal{G}}$ at every point. The second follows from the formula
$$\mathcal{O}(D_{tan})=T^{*}_{\mathcal{G}}\otimes N_{\mathcal{F}}$$(see for example \cite[Lemme 4]{Br1}). We know that $N_{\mathcal{F}}$ is trivial so we get that $T^{*}_{\mathcal{G}}$ is also trivial.
\end{proof}
\begin{lemma} \label{60}
The canonical bundle $K_{S}$ of $S$ is isomorphic to 
$\pi^*(K_{C})\otimes {\mathcal O} (\sum(l_{E}-1)E)$, where $l_E$ denotes the multiplicity of $E$ as 
a fiber of $\pi$.
\end{lemma}
\begin{proof} We have just seen that $K_\mathcal{G}$ is trivial, on the other hand,
one can compute this as the canonical bundle of a foliation defined by a fibration:
$$K_{\mathcal{G}}\simeq K_{S/C}\otimes {\mathcal O} (\sum(1-l_{E})E)$$ where $K_{S/C}=K_{S}\otimes \pi^*(K_{C}^*)$ is the relative canonical bundle and the sum is over all fibres of $\pi$. 
By Kodaira's canonical bundle formula $K_{S}$ is equal to $\pi^*(K_C\otimes L)\otimes {\mathcal O}(\sum(l_{E}-1)E)$ where $L=(R^{1}\pi_{*}\mathcal{O}_{S})^{*}\cong \pi_{*}K_{S/C}$ is the 
so-called 
fundamental line bundle   on $C$ (see for example \cite[Corollary V.12.3]{BPV}). So  $T^{*}_{\mathcal{G}}\simeq K_{S/C}\otimes {\mathcal O}(\sum(1-l_{E})E) \simeq \pi^*(L).$ Since $T^{*}_{\mathcal{G}}\simeq \mathcal{O}$ by the Corollary \ref{50}, we get that $L$ is trivial. Hence $K_{S}\simeq \pi^*(K_{C})\otimes {\mathcal O}(\sum(l_{E}-1)E)$. 
\end{proof}

Note that the triviality of the fundamental line bundle $(R^{1}\pi_{*}\mathcal{O}_{S})^{*}$ also has the following geometric consequence (see e.g. \cite{FM}[Proposition 3.22]):
\begin{proposition} \label{16}
Let $S$ be an elliptic surface and let $d=\operatorname{deg}(R^{1}\pi_{*}\mathcal{O}_{S})^{*}$ and $g$ be the genus of the base curve C. Then $p_{g}(S)=d+g-1$, if $(R^{1}\pi_{*}\mathcal{O}_{S})^{*}$ is not trivial, and $p_{g}(S)=g$ if $(R^{1}\pi_{*}\mathcal{O}_{S})^{*}$ is trivial.
\end{proposition}
\begin{corollary} \label{70}
The geometric genus of $S$ is equal to $g$.
\end{corollary}

In particular we get by the Lemma \ref{15} that $g=g(C)>0$. 


\section{Geometry of deformations of $S$}

Let us denote by $M$ an irreducible component of Barlet space of $X$ containing $[S].$
The following lemma follows from the fact that the deformations of lagrangian 
subvarieties are unobstructed and remain lagrangian (e.g.  \cite{V}).

\begin{lemma} \label{54}
(i)The Barlet space $\mathcal{B}(X)$ is smooth of dimension $g+1$ near $S$. In particular, $S$ is contained in a unique irreducible component $M$ of $\mathcal{B}(X)$. 

(ii)If $[S^{'}]\in M$ represents a smooth subvariety $S'$, then $S'$ is an elliptic Lagrangian surface in $X$ and $h^0(\Omega^1_{S^{'}})=g(C)+1.$
\end{lemma}
\begin{proof}
(i) See the proof of \cite{GLR}[Lemma 3.1]. The tangent space to $M$ at $[S]$ is isomorphic to $H^{0}(N_{S/X}).$ Since $S$ is Lagrangian, the symplectic form induces an isomorphism $\Omega^1_{S}\simeq N_{S/X},$ so $T_{[S]}(M)\simeq H^{0}(\Omega^1_{S}).$ The dimension of this last vector space is equal to $p_g(S)+1-\chi({\mathcal O}_S)=g+1.$

(ii) By \cite{V}, deformations of $S$ in $X$ are unobstructed and are also Lagrangian. So $S'$ is Lagrangian and the symplectic form induces an isomorphism $\Omega^1_{S^{'}}\simeq N_{S^{'}/X}$. The dimension of the tangent space to $M$ at $S$ is equal to the dimension of the tangent space to $M$ at $S'$, i.e.  $h^0(\Omega^1_{S^{'}})= h^0(\Omega^1_{S})=g(C)+1.$ 
The Kodaira dimension is invariant by smooth deformations (even the plurigenera are,
by a theorem of 
Siu \cite{S}), so the Kodaira dimension of $S'$ is equal to one, i.e. $S^{'}$ is a properly elliptic surface over a curve $C'$. 
\end{proof}

We also have $c_{2}(S^{'})= c_2(S)=0.$  Hence $S'\to C'$ has only multiple fibers as singularities, since all singular fibers give nontrivial contribution to $c_2.$ The number of multiple fibers and their multiplicities are also invariant under deformation, see for example \cite[Proposition 7.1]{FM}. The genus of $C'$ is equal to $g=g(C)$, since a deformation of $S$ induces a deformation of $C$ (see e.g. \cite[Theorem 7.11]{FM}).


\begin{corollary} \label{17}
The geometric genus $p_{g}(S')$ of $S'$ is equal to $g(C),$ where $C$ is the base curve of $S.$ Moreover the fundamental line bundle of the elliptic fibration on $S^{'}$ is trivial.
\end{corollary}
\begin{proof}
The first statement is the invariance of (pluri)genera; the second follows by $g(C)=g(C')$ and 
Proposition \ref{16}.
\end{proof}
Now we are ready for our main conclusion on deformations of $S$.
\begin{lemma} \label{13}
The sheaf $\Omega^1_{S^{'}}$ decomposes into the direct sum $$\Omega^1_{S^{'}}\simeq \mathcal{O}\oplus (\pi^*(\Omega^{1}_{C^{'}}) \otimes {\mathcal O}(\sum(l_{E}-1)E)),$$ where $\pi: S^{'}\to C'$ is an elliptic fibration of $S'$. Since $\Omega^1_{S^{'}}\simeq N_{S^{'}/X}$ the same statement is true for the normal bundle to $S'$ in $X.$
\end{lemma}
\begin{proof} 
The elliptic fibration on $S'$ induces a smooth foliation $\mathcal{G'}$.
Consider the corresponding exact sequence on $S'$: 
\begin{equation} \label{18}
0\to N_{\mathcal{G'}}^{*} \to \Omega^{1}_{S^{'}}\to K_{\mathcal{G'}} \to 0. 
\end{equation}
First of all $K_{\mathcal{G'}}$ is trivial since it is isomorphic to the pull-back of the fundamental line bundle that is trivial by the Corollary \ref{17}, so $K_{S^{'}}\simeq K_{\mathcal{G'}}\otimes N_{\mathcal{G'}}^{*}\simeq N_{\mathcal{G'}}^{*}.$ Note that $N_{\mathcal{G'}}^{*}$ is isomorphic to $\pi^*(\Omega^{1}_{C^{'}})\otimes {\mathcal O} (\sum(l_{E}-1)E)$ and the exact sequence (\ref{18}) has the following form:
\begin{equation} \label{14}
0\to \pi^*(\Omega^{1}_{C^{'}})\otimes {\mathcal O}(\sum(l_{E}-1)E) \to \Omega^{1}_{S^{'}}\to \mathcal{O} \to 0.
\end{equation}
The long exact sequence of cohomologies starts as $$0\to H^{0}(\pi^*(\Omega^{1}_{C^{'}})\otimes \sum(l_{E}-1)E) \to H^{0}(\Omega^{1}_{S^{'}})\to H^0(\mathcal{O})\to \operatorname{Ext}^{1}(\mathcal{O},\pi^*(\Omega^{1}_{C^{'}})).$$  

But $h^{0}(\Omega^{1}_{S^{'}})=g(C)+1$ by Lemma \ref{54}(ii), and $h^{0}(\pi^*(\Omega^{1}_{C^{'}})\otimes \sum(l_{E}-1)E)=g(C)$ since $h^{0}(\pi^*(\Omega^{1}_{C^{'}})\otimes \sum(l_{E}-1)E)=h^{0}(N_{\mathcal{G}}^{*})=h^{0}(K_{S^{'}})$ and the last number is equal to $g(C')=g(C)$ by Corollary \ref{17}. The map $H^{0}(\Omega^{1}_{S^{'}})\to H^0(\mathcal{O})$ must then be surjective, hence the exact sequence (\ref{14}) splits.  
\end{proof}

\section{Conclusion by deformation theory} 
Let us denote by $\mathcal{U}$ the universal family over the component $M$ of the Barlet space of $X$ containing the parameter point of $S$,  and by $\varepsilon$ and $\gamma$ the natural maps of $\mathcal U$ to $X$ and $M.$ It is well-known that $M$ is compact 
and $\varepsilon$ and $\gamma$ are proper. The tangent space to $\mathcal{U}$ at a point $(u,S^{'}),$ where $u\in S^{'}$ and $S^{'}$ is smooth, can be described as the first term of the following exact sequence: $$0\to T_{\mathcal{U},{(u,S^{'})}}\to T_{X,{u}}\oplus H^{0}(N_{S^{'}/X})\to (N_{S^{'}/X})_{u}, $$ where the last map sends couple $(v,s),$ where $v\in T_{X_{u}}$ and $s\in H^{0}(N_{S^{'}/X}),$ to $(v\operatorname{mod}{T_{S^{'}}})-s(u).$ The differentials of $\varepsilon$ and $\gamma$ are compositions of the first morphism with the projections to $T_{X_{u}}$ and $H^{0}(N_{S^{'}/X})$ respectively. Note that we can compute the dimension of the kernel of $d\varepsilon $ at every smooth point $(u,S^{'})$: it is equal to the dimension of the space of sections
of $N_{S^{'}/X}$ vanishing at $u$, i.e. $g-1$ if $u$ does not lie on a multiple fiber on $S^{'}$, and $g$ if $u$ does. Consider the case of a point $(u^{'},S^{'})$ of $\mathcal{U}$ such that $u^{'}$ does not lie on a multiple fiber on $S^{'}.$ Let us prove that fibers of $\varepsilon$ over $u^{'}$ are smooth reduced varieties of dimension $g-1$: indeed a fiber has dimension at least $g-1$ since the dimension of $\mathcal{U}$ is equal to $g+3$, so the dimension of the fiber and that of the tangent space to the fiber are both 
equal to $g-1$ by the remark above, so that the fiber is equidimensional, reduced and 
smooth. 
Moreover we get that $d\varepsilon $ is surjective at $(u^{'},S^{'})$ so $\varepsilon$ is a submersion on the preimage $U$ of an analytic neighborhood of $u^{'}$ that does not contain points of multiple fibers on $S^{'}.$ Hence by Ehresmann's lemma $\varepsilon$ is a locally trivial fibration on $U$.
\begin{lemma} \label{31} (cf. \cite{A})
Whenever $L\cap K\ne \emptyset$ where $L,K\in M$ and $L$ is smooth, the intersection $L\cap K$ is equidimensional. That is, since $c_2(N_{L/X})=0$, all irreducible components of $L\cap K$ are curves.
\end{lemma}
\begin{proof}
Consider $\varepsilon^{-1}(L)$ -- the preimage of $L$ in $\mathcal{U}.$ Consider one of the irreducible components $\widetilde{\varepsilon^{-1}(L)}$ of $\varepsilon^{-1}(L).$ By the remarks above the dimension of $\widetilde{\varepsilon^{-1}(L)}$ is equal to $g+1.$ 
Note that $\gamma$ can not be dominant on $\widetilde{\varepsilon^{-1}(L)}.$ Indeed suppose that $\gamma$ is dominant on $\widetilde{\varepsilon^{-1}(L)}$, then, since the dimension of $M$ is also equal to $g+1,$ the dimension of a general fibre would be zero. If 
some components dominate $M$ and other don't, we pick a point in $M$ outside of images 
of those non-dominating components and obtain that the corresponding surface has purely
 zero-dimensional non-empty intersection with $L$, 
which is impossible since $c_{2}(N_{L/X})=0$ and this is the same as the 
self-intersection number of $L$ in $X$.
 So all components are non-dominating and the dimension of a general fiber is always 
equal to 
$1.$ Since this fiber is exactly the intersection of $L$ and the corresponding surface
from the family, we are done. 
\end{proof}

The following observation is well-known from deformation theory; it is a consequence
of the tubular neighbourhood lemma and the unobstructedness of deformations of 
lagrangian subvarieties.

\begin{obs}\label{smalldefo} 
The locus where a smooth lagrangian surface $L$ from the family $M$
intersects its small deformation is the zero locus of the corresponding section 
$s\in H^0(N_{L/X})\simeq H^0(\Omega^1_L)$. The intersection is transversal wherever 
the section $s$ has no multiple zero.
\end{obs}

We thus have two cases to consider: when the the genus $g$ of the base curves of our
elliptic surfaces from the family $M$ is greater 
than $1$, any elliptic fiber on $L$ is in the zero locus of some section of $\Omega^1_L$
and therefore is a component of the intersection of $L$ with a small deformation;
in the case when $g=1$, such intersections happen only along multiple fibers since
the sections of $\Omega^1_L$ then have no other zeros.

By the same type of reasons, we have the following 

\begin{lemma}\label{dominant}
The map $\varepsilon$ is dominant.
\end{lemma}

Indeed the surface $S$ deforms away from the divisor $D$ as it is seen on the 
infinitesimal level, so there is a surface from the
family $M$ through a general point of $X$.

If we assume that $S$ is not a torus, the family of intersections of surfaces 
from $M$ is also dominant:

\begin{lemma}\label{nofibrat}
If $S$ is not a torus, then there is more than one $S$ through a 
general point of $X$.
\end{lemma}

\begin{proof} Assume the contrary, then the family $M$ gives a rational fibration of $X$ in subvarieties of non-maximal Kodaira dimension.  By \cite{AC-fib} it must be a fibration on lagrangian tori, a contradiction.
\end{proof}

At a general point of $x\in X$, the general intersections are transversal as 
follows from Sard's lemma. More 
precisely, consider in $\mathcal{U}\times \mathcal{U}$ the complement to the preimage of the diagonal $\Delta_{M}$ in $M\times M,$ let us denote it by $\mathcal{U}\times \mathcal{U}\setminus ((\gamma\times \gamma)^{*}(\Delta_{M})).$ Consider the preimage of $\Delta_{X}$, the diagonal in $X\times X$, in the following diagram:
$$
\begin{CD}
\mathcal{U}\times \mathcal{U}\setminus ((\gamma\times \gamma)^{*}(\Delta_{M})) @>\varepsilon \times \varepsilon>> X\times X \\
@V{\gamma\times \gamma}VV \\
M\times M \setminus (\Delta_{M})
\end{CD}
$$ 
Denote this preimage by $I.$ Over $M\times  M \setminus (\Delta_{M})$,
$I$ is the family of intersections of surfaces from $M.$ At a general point 
$x\in X$, consider the intersection of smooth surfaces corresponding to $m_1$ and $m_2\in M$; the point $((x, m_1), (x, m_2))$ is a general point of $U$. As 
the intersections dominate $X$, one deduces from Sard's lemma that 
$\varepsilon$ 
is submersive at both $(x, m_1)$ and $(x, m_2)$; then $\varepsilon\times \varepsilon$ is submersive at $((x, m_1), (x, m_2))$, so that $I$ is smooth and the intersection is transversal.

With this in mind, let us first treat the case $g>1$.

\begin{proposition} \label{23} Assume $g>1$.
Let $x$ be a general point of $X$ and $S$ a general (smooth) surface from the family $M$ through this point. Let $E$ be the fiber of the elliptic fibration 
on $S$ through $x$. Then any other surface from the family $M$ passing
through $x$ intersects $S$ along $E$ (and nothing else) in the neighbourhood 
of $x$.
\end{proposition}
\begin{proof}
There are certainly surfaces satisfying this condition, namely suitable 
small deformations of $S$ mentioned in observation \ref{smalldefo}. Take one of them, say $S'$,
intersecting $S$ transversally at $x$.
Suppose there are other intersection curves through $x$; consider a general
$L$ from $M$ such that the intersection component with $S$ is not $E$.
We apply the same argument as in \cite{A}: three tangent planes $T_xS$, $T_xL$,
$T_xS'$ intersect along three distinct lines and generate a three-space $H_x$ in 
$T_xX$; but these are lagrangian planes so we get a contradiction (each of 
the intersection lines must then be orthogonal to $H$ with respect to the 
symplectic form).
\end{proof}

\begin{corollary} 
The case $g>1$ is impossible.
\end{corollary}

\begin{proof}
Indeed we see that all surfaces passing through the general point $x\in X$ intersect along the elliptic fiber 
of one of them; $X$ must then be rationally fibered in elliptic curves, as there is thus only one such a fiber
through $x$. This is 
impossible by \cite{AC-fib}.
\end{proof}

It remains to rule out the case when $g=1$ but $S$ is not a torus, that is, $\pi:S\to C$ has multiple fibers. The difficulty is now that, as we shall see,
 the multiple fibers of all our surfaces are contained in the divisor $D$; on
the other hand, $S$ now intersects its small deformation $S'$ necessarily along
a multiple fiber, so such an intersection is not generic. To deal with this,
we go into more details on deformations of elliptic fibers.

So consider the case where $S$ is an elliptic surface with an elliptic base curve and at least one multiple fiber. By lemma \ref{13} for a smooth surface $S'$ from the family $M$ we have $\Omega^{1}_{S^{'}}\simeq {\mathcal O}(\sum(l_{E}-1)E)\oplus \mathcal{O},$ so $M$ is 2-dimensional and the surfaces from $M$ still cover $X.$ Denote by $\mathcal{D}(E)$ the irreducible component of the Douady space that contains a fiber $E$ of the elliptic fibration of $S$ and by $\mathcal{D}(E)_{mul}$ an irreducible component of the Douady space that contains the reduction of a multiple fiber of $S$ (all fibers are in the same component, but there might
be several components parameterizing the reductions of fibers, e.g. of 
different 
multiplicities). 
\begin{lemma} \label{90}
$\mathcal{D}(E)$ is 3-dimensional and the normal bundle $N_{E/X}$ for each
(possibly multiple) fiber $E$ of the elliptic fibration of a surface $S'$ from $M$ is isomorphic to $\mathcal{O}\oplus \mathcal{O}\oplus \mathcal{O}.$ In particular locally in a neighbourhood $U$ of $E$ one has a fibration of $U$ by small deformations of $E.$
\end{lemma}
\begin{proof}
Consider the exact sequence
$$0\to N_{E/S^{'}}\to N_{E/X}\to N_{S^{'}/X}|_{E}\to 0,$$
we have $N_{E/S^{'}}\simeq \mathcal{O}$ and $N_{S^{'}/X}|_{E}\simeq \mathcal{O}\oplus \mathcal{O},$ so
\begin{equation} \label{34}
0\to \mathcal{O} \to  N_{E/X}\to \mathcal{O}\oplus \mathcal{O} \to 0
\end{equation}
Furthermore $h^{0}( N_{E/X})$ is the tangent 
 to $\mathcal{D}(E)$ and must therefore be at least three-dimensional; 
we conclude that $h^{0}( N_{E/X})=3$ and the exact sequence \ref{34} splits.
The last statement follows from the standard calculation of the differential 
of the
evaluation map from the universal family of elliptic curves to $X$.
\end{proof}
\begin{lemma}
$\mathcal{D}(E)_{mul}$ is 1-dimensional and all elliptic curves from the family $\mathcal{D}(E)_{mul}$ are contained in 
$D,$ more precisely they are fibers of the elliptic surfaces from our family $M$ contained in $D$. 
\end{lemma} \label{35}
\begin{proof} All multiple fibers are in $D$ since by an easy calculation, ${\mathcal O}_{S'}(D)={\mathcal O}(\sum_E(m(E)-1)E)$, where $m$ stands for multiplicity. They must coincide with fibers of the elliptic surfaces in $D$ since otherwise we would have surfaces from $M$ intersecting at points and not at curves, contradicting $c_2(S')=0$.
Let $E_{m}$ be the reduction of a fiber of multiplicity $m$ on $S^{'}$. Consider the exact sequence $$0\to N_{E_{m}/S^{'}}\to N_{E_{m}/X}\to N_{S^{'}/X}|_{E_{m}}\to 0.$$ Since $N_{E_{m}/S^{'}}$ is a torsion bundle of order $m$
(see e.g. \cite{FM}[Lemma III.8.3 ]) and $N_{S^{'}/X}|_{E_{m}}$ is isomorphic to the direct sum of the trivial bundle and the inverse of that 
torsion bundle of order $m$ (see \cite{FM}[Lemma III.8.3 ]), we get that $\operatorname{dim}(\mathcal{D}(E)_{mul}) \le 1$. Since multiple fibers of smooth surfaces from $D$ already form a 1-dimensional family, we get that  $\operatorname{dim}(\mathcal{D}(E)_{mul}) = 1$.
\end{proof}
\begin{corollary} \label{36} 
A smooth curve from $\mathcal{D}(E)$ either does not intersect $D$ or lies in $D.$ In particular any smooth surface from $M$ intersects any surface from $D$ by fibers of the corresponding elliptic fibrations. 
\end{corollary} 
\begin{proof} A general small deformation of an elliptic curve in $D$ does not intersect $D$, so the intersection number
$DE$ is zero. An elliptic curve from $\mathcal{D}(E)$ contained in $D$ must be a fiber of some $S$ by the same reason as in the preceding lemma.

\end{proof}

Now we are ready to finish the proof of our main result. 

\medskip

{\it End of proof of the main theorem in the case $g=1$:} As we have already seen studying the case $g>1$, to get a contradiction
it is sufficient to show that all surfaces from $M$ passing through a general $x\in X$ intersect at the
same elliptic curve $E$ which is the fiber of the elliptic fibration on each of them. 
Suppose that this is not true: the intersection of two smooth surfaces $L$ and $S$ from $M$ contains a curve 
$C$ that is not a fiber of the corresponding elliptic fibrations. We claim that $L$ and $S$ share a multiple 
fiber (which automatically intersects the multisection $C$ of $L$). Indeed let $F$ be a multiple fiber of $L$ and 
let $c$ be a point of intersection of $F$ and $C.$ We have seen that $F$ is contained in $D$. Consider the 
elliptic curve $E$ on $S$ that contains $c.$ Since $F\subset D$, $E$ intersects $D$ and by the Corollary \ref{36} 
$E\subset D,$ i.e. $E$ is equal to $F.$ 

But this is impossible since no neighbourhood of $E$ is then fibered in elliptic curves from $\mathcal{D}(E)$
(as some curves around $E$, those lying on $L$ and those lying on $S$, must intersect at the points of $C$). 

\medskip


{\bf Acknowledgement} The first-named author was partially supported by an RSCF grant 14-21-00053 within the Laboratory of Algebraic Geometry at the National Research University Higher School of Economics.

\end{document}